\documentclass[11pt,a4paper]{amsart} 
\calclayout
\usepackage{microtype}
\usepackage[english]{babel}
\usepackage[maxbibnames=5,minbibnames=4,backend=bibtex]{biblatex}
\usepackage{csquotes}
\usepackage{amssymb}
\usepackage{amsmath}
\usepackage{amsthm}

\newtheorem{thm}{Theorem}[section]

\newtheorem{prop}[thm]{Proposition}
\newtheorem{coro}[thm]{Corollary}

\theoremstyle{remark}
\newtheorem{rema}[thm]{Remark}
\newtheorem{exa}[thm]{Example}
\newtheorem{defi}[thm]{Definition}

\DeclareMathOperator\ad{ad}

\newcommand{\rank} 
{\operatornamewithlimits{rank}}

\usepackage{MnSymbol}
\usepackage{stmaryrd}
\usepackage{slashed}
\usepackage{tikz}
\usepackage[all,cmtip]{xy}
\usepackage{todonotes}

\title{Leaves of stacky Lie algebroids}
\author{Daniel \'Alvarez}
\address{Instituto de Matem\'atica Pura e Aplicada, Estrada Dona Castorina, 110, Jardim Bot\^anico, CEP \texttt{22460320},
Rio de Janeiro, Brasil}
\email{uerbum@impa.br}
\date{} 
\begin{document}
\begin{abstract} We show that the leaves of an LA-groupoid which pass through the unit manifold are, modulo a connectedness issue, Lie groupoids. We illustrate this phenomenon by considering the cotangent Lie algebroids of Poisson groupoids thus obtaining an interesting class of symplectic groupoids coming from their symplectic foliations. In particular, we show that for a (strict) Lie 2-group the coadjoint orbits of the units in the dual of its Lie 2-algebra are symplectic groupoids, meaning that the classical Kostant-Kirillov-Souriau symplectic forms on these special coadjoint orbits are multiplicative.  
\end{abstract}
\maketitle
\tableofcontents 
 
\section{Introduction}\label{sec:intro} Lie algebroids are a unifying concept in differential geometry: they allow us to describe foliations, Lie algebra actions, connections on principal bundles and Poisson brackets among many other examples. Lie algebroids can also be seen as the infinitesimal objects associated to Lie groupoids, generalizing the correspondence between Lie groups and Lie algebras. Lie groupoids, on the other hand, provide geometric models for singular spaces when regarded as atlases for differentiable stacks, see \cite{difger}. Lie algebroids over differentiable stacks can be described by {\em LA-groupoids} \cite{macdou} which are essentially groupoid objects in the category of Lie algebroids and so they can be called ``stacky Lie algebroids'', see \cite{walphd} for a precise statement. To every Lie algebroid is associated a singular foliation on its base manifold. We shall see that, at least up to a connectedness issue, the leaves of the singular foliation associated to an LA-groupoid which pass through the unit manifold inherit themselves a Lie groupoid structure, see Theorem \ref{orblagr}.

In order to illustrate Theorem \ref{orblagr} we shall focus on Poisson structures. A Poisson structure on a manifold determines a singular symplectic foliation which completely characterizes it \cite{weipoi}. From the viewpoint of Lie theory, Poisson structures arise as the infinitesimal counterparts of Lie groupoids endowed with a compatible symplectic structure, which are called {\em symplectic groupoids} \cite{weisygr,karpoi}. Whenever a Poisson structure is induced by a symplectic groupoid, it is called {\em integrable}; but unlike Lie algebras, not every Poisson manifold is integrable. Symplectic groupoids are important tools in the study of Poisson manifolds but in general they are difficult to construct, see e.g. \cite{craruipoi}.

A {\em Poisson groupoid} \cite{weicoi} is a Lie groupoid endowed with a Poisson structure which is compatible with the groupoid multiplication: these objects generalize on the one hand symplectic groupoids and, on the other, they generalize Poisson-Lie groups \cite{driham}. We shall see in Proposition \ref{cotpoigro} that those leaves of the symplectic foliation associated to a Poisson groupoid structure which pass through the unit manifold provide examples of symplectic groupoids. 

We exemplify this observation by means of (strict) Lie 2-groups \cite{grp}, which are group objects in the category of Lie groupoids. Lie 2-groups equipped with a Poisson structure compatible with both the group and groupoid multiplications are called {\em Poisson 2-groups} \cite{poi2gro}. The symplectic leaves of Poisson groups are the orbits of the {\em infinitesimal dressing actions} \cite{semdres}; for Poisson 2-groups these dressing actions are multiplicative actions, meaning that the action maps are Lie groupoid morphisms. The simplest example of a Poisson 2-group is a Lie 2-group endowed with the zero bracket. In this case, the dressing action coincides with the coadjoint action on the dual of its Lie algebra so we get, in particular, that the coadjoint orbits of Lie 2-groups which contain a unit, endowed with their canonical Kostant-Kirillov-Souriau symplectic forms, are symplectic groupoids.  
 
\section{Preliminaries} \subsection{Lie groupoids and Lie algebroids} A {\em smooth groupoid} is a groupoid object in the category of not necessarily Hausdorff smooth manifolds such that its source map is a submersion. A {\em Lie groupoid} is a smooth groupoid such that its base and source-fibers are Hausdorff manifolds; see e.g. \cite{moeint} for an introduction to these concepts. We denote the source, target, multiplication, unit and inversion maps of a Lie groupoid $M_1$ over $M_0$, $M_1 \rightrightarrows M_0$, as $\mathtt{s}$, $\mathtt{t}$ and $\mathtt{m}$, $\mathtt{u}$, $\mathtt{i}$ respectively. In order to avoid ambiguity when dealing with several groupoids we use a subindex $\mathtt{s}=\mathtt{s}_{M_1}$, $\mathtt{t}_{M_1}$, $\mathtt{m}_{M_1}$ to specify the groupoid under consideration. A {\em Lie algebroid} consists of a vector bundle $A $ over a manifold $M_0$ endowed with a vector bundle map $\mathtt{a}:A \rightarrow TM_0$, called the {\em anchor}, and a Lie algebra bracket on $\Gamma (A)$ that satisfies the Leibniz rule: $[X,fY]=f[X,Y]+(\mathcal{L}_{\mathtt{a}(X)}f)Y$ for all $X,Y\in \Gamma (A)$ and all $f\in C^\infty(M)$. Given a Lie groupoid $M_1 \rightrightarrows M_0$, there is a canonical way to endow $\ker T \mathtt{s}|_{M_0}$ with a Lie algebroid structure over $M_0$, see \cite{moeint}; we call it the {\em Lie algebroid of $M_1 \rightrightarrows M_0$} and we denote it by $A_{M_1}$.  \subsection{LA-groupoids and double Lie groupoids} 
 \begin{defi}\cite{macdou} An {\em LA-groupoid} is a Lie groupoid in the category of Lie algebroids, that is, a Lie groupoid $A_1 \rightrightarrows A_0$ where $A_1$ and $A_0$ are Lie algebroids, the structure maps are Lie algebroid morphisms over the structure maps of a base groupoid $M_1 \rightrightarrows M_0$, and the map $A_1 \rightarrow \mathtt{s}_{M_1}^*A_0 $ induced by $\mathtt{s}_{A_1}$ is surjective. \end{defi}

	\begin{exa} Let $M_1\rightrightarrows M_0$ be a Lie groupoid. In general, $TM_1\rightrightarrows TM_0 $ (with the structure maps given by applying the tangent functor) and the usual Lie bracket is an LA-groupoid. \end{exa}
Let $M_1\rightrightarrows M_0$ be a Lie groupoid. Then $T^*M_1\rightrightarrows A_{M_1}^*$ (where $A_{M_1}\rightarrow M_0$ is the Lie algebroid of $M_1\rightrightarrows M_0$) is also a Lie groupoid, called the {\em cotangent groupoid} of $M_1\rightrightarrows M_0$. The multiplication $\hat{\mathtt{m}}$ in $T^*{M_1}$ is characterized by the following property:
\begin{align} \langle \hat{\mathtt{m}}(\xi,\eta),T{\mathtt{m}}(u,v)\rangle =\langle \xi,u\rangle+\langle \eta,v\rangle, \label{eq:cotgro}\end{align} 
for all composable $u\in T_xM_1$, $v\in T_yM_1$ and $\xi\in T^*_xM_1$, $\eta\in T^*_yM_1$ \cite{cosgro}.

LA-groupoids are the infinitesimal counterparts to the following higher categorical structures.
\begin{defi}\cite{browmac,macdou} A double topological groupoid is a groupoid object in the category of topological groupoids. A double topological groupoid is represented as a diagram of the form
\[ \xymatrix{ {G}_1\ar@<-.5ex>[r] \ar@<.5ex>[r]\ar@<-.5ex>[d] \ar@<.5ex>[d]& {M}_1 \ar@<-.5ex>[d] \ar@<.5ex>[d] \\ G_0 \ar@<-.5ex>[r] \ar@<.5ex>[r] & M_0. }\]  
A {\em double Lie groupoid} is a double topological groupoid as in the previous diagram such that: (1) each of the side groupoids is a smooth groupoid, (2) $M_1$ and $G_0$ are Lie groupoids over $M_0$, and (3) the double source map $(\mathtt{s}^{M_1},\mathtt{s}^{G_0}):  G_1 \rightarrow M_1 \times_{M_0} G_0 $ is a submersion (the superindices $\quad^{M_1},\quad^{G_0}$ denote the groupoid structures $G_1 \rightrightarrows M_1$, $G_1 \rightrightarrows G_0$ respectively). \end{defi}
It is known that double Lie groupoids differentiate to LA-groupoids \cite{macdou}, but the integration problem is not fully understood, see \cite{morla}. The next simple example shows that the na\"ive attempt to integrate an LA-groupoid by applying Lie's second theorem to all its structure maps fails in general. Other integrations of the Lie algebroids involved can still carry a double Lie groupoid structure, the best candidate being an integration whose vertical source-fibers are 1-connected in the stacky sense, see \cite{funsta}.  
\begin{exa}\label{nonintlagr} Consider a free action of $\mathbb{R}^2$ on a manifold $M$ with at least three different orbits $O_i$. Consider $x_i\in O_i$ for $i=1,2,3$ and $v\in \mathbb{R}^2$ with $|v|>1$ and take 
\[ N=M-(\{v\cdot x_1\}\cup\{z\cdot x_3|z\in \mathbb{S}^1\}). \] 
On $N$ there is an infinitesimal action of $\mathbb{R}^2$ which induces a diagonal action on $N^2$; the image of this action is a distribution $D$ which is a subgroupoid of the tangent groupoid $T N^2 \rightrightarrows TN$ with respect to the pair groupoid structure on $N^2 \rightrightarrows N$ and hence it is an LA-groupoid. 

The leaf of $D$ through $(x_1,x_2)$ is diffeomorphic to $\mathbb{R}^2$ with a point removed and the leaves through $(x_1,x_3)$ and $(x_2,x_3)$ are diffeomorphic to a disk. Call $\mathcal{O}_{ij}$ the leaf through $(x_i,x_j)$. We have that $\pi_1\left(\mathcal{O}_{12},(x_1,x_2)\right)\cong \mathbb{Z} $. Take a non trivial homotopy class $a$ in that group. Then $\mathtt{s}(a)\in \pi_1(O_2)=1$ is the trivial class so $a$ and the constant path based on $(x_2,x_3)$ should be composable. However, there is no element in $\pi_1(\mathcal{O}_{13},(x_1,x_3))=1$ which projects to $\mathtt{t} (a)\in \pi_1(O_1-\{v\cdot x_1\})$ which is non trivial. So the monodromy groupoid of $D$ carries no compatible groupoid structure over the monodromy groupoid of $D|_N$. \end{exa}

\section{The main result} \begin{defi} Let $A_1 \rightrightarrows A_0$ be an LA-groupoid over $M_1 \rightrightarrows M_0$. Let $\mathcal{O}\subset M_0 $ be an $A_0$-orbit and take $\mathcal{O}'$ the $A_1$-orbit which contains $\mathcal{O} $. Let us denote by $\widehat{\mathcal{O} }$ the union of the connected components of $\mathcal{O}' \cap \mathtt{s}_{M_1}^{-1}(x)$ which contain a unit for all $x\in \mathcal{O}$. \end{defi}  
\begin{thm}\label{orblagr} Let $A_1 \rightrightarrows A_0$ be an LA-groupoid over $M_1 \rightrightarrows M_0$ and let $\mathcal{O}\subset G_0 $ be an $A_0$-orbit. Then $\widehat{\mathcal{O} }$ is an immersed Lie subgroupoid of $M_1$ over $\mathcal{O}$. \end{thm}

What motivates our main result is the following algebraic observation. If an LA-groupoid is integrable by a double Lie groupoid, then we can use this integration to show that the orbits through the units inherit a smooth groupoid structure.
\begin{prop}\label{orbdougro} Let $G_1 $ be a double Lie groupoid with sides $M_1$ and $G_0 $ over $M_0$. Then the $G_1 $-orbits of the units in $M_1 $ are immersed smooth subgroupoids over the corresponding $G_0$-orbits in $M_0$. \end{prop}  
\begin{proof} Let $g,h \in G_1 $ be such that $\mathtt{s}^{M_1}(g)=\mathtt{s}^{M_1}(h)=u\in M_0$ and $\mathtt{t}^{G_0}(g)=x$, $\mathtt{t}^{G_0}(h)=y$ are composable, i.e. $\mathtt{s}_{M_1}(x)=\mathtt{t}_{M_1}(y)$. Then $k=\mathtt{m}^{M_1}\left( \mathtt{i}^{M_1}(\mathtt{s}^{G_0}(g)),\mathtt{t}^{G_0}(h)\right)$ is defined and so is $g'=\mathtt{m}^{M_1}(g,k)$. Since $\mathtt{s}^{G_0}(g')=\mathtt{t}^{G_0}(h)$ we can take $l=\mathtt{m}^{G_0}(g',h)$ and then $\mathtt{s}^{M_1}(l)=u$, $\mathtt{t}^{M_1}(l)=\mathtt{m}^{G_0}(x,y)$. So these orbits are subgroupoids. In general, the orbits of a smooth groupoid are immersed submanifolds even when the base is non Hausdorff. \end{proof} 
In general, we have to use an infinitesimal argument in order to prove the result.
 
	\begin{proof}[Proof of Theorem \ref{orblagr}] Let $\mathcal{O}$ be an $A_0$-orbit and let $\mathcal{O}'$ be the $A_1$-orbit which contains $\mathcal{O} $. Let us denote by $\mathtt{s}_{M_1}=\mathtt{s} $ and $\mathtt{t}_{M_1}=\mathtt{t} $ for simplicity in what follows. Recall that $A_{M_1}=\ker T\mathtt{s}|_{M_0}$ is the Lie algebroid of $M_1 \rightrightarrows M_0$ and $\mathtt{a}=T \mathtt{t}|_{A_{M_1}} $ is its anchor. This proof consists of the following steps: 
\begin{enumerate} \item we construct a Lie subalgebroid $B$ of the tangent Lie algebroid of $M_1 \rightrightarrows M_0$ over $\mathcal{O} $ such that the span of the right-invariant vector field in $M_1$ induced by its inclusion in $A_{M_1}$ coincides with $(T \mathcal{O}'\cap \ker T \mathtt{s})|_{\widehat{\mathcal{O}}}$;
\item using Lie's second theorem we show that $\widehat{\mathcal{O} }$ is the image of a Lie groupoid morphism $\Psi:H \rightarrow M_1$, where $H$ is the source-simply-connected integration of $B$. Since $\Psi$ is the identity on $\mathcal{O} $, we see that $\widehat{\mathcal{O} }=\Psi(H)$ inherits a groupoid structure over $\mathcal{O} $. Let us stress that our reliance on Lie's second theorem implies the existence of a groupoid structure only on $\widehat{\mathcal{O}}$, which can be seen as the ``source-connected component'' of $\mathcal{O}'$.
 \end{enumerate} 

{\em Step 1}. Since $\mathtt{s}_{A_1}$ is fiberwise surjective, the pullback of $A_1 $ to $M_0$ splits as the sum $A_1|_{M_0}=\ker \mathtt{s}_{A_1}|_{M_0}\oplus A_0$. 

By definition, we have that $T_x \mathcal{O}'=\mathtt{a}_1(A_1|_x)$ for all $x\in \mathcal{O} $. But the anchor map $\mathtt{a}_1$ of $A_1$, being a groupoid morphism, induces a square of morphisms of vector bundles over $G_0$:
		\[ \xymatrix{ \ker \mathtt{s}_{A_1}|_{M_0}\ar[d]_{\mathtt{a}_1|}  \ar[r]^{\mathtt{t}_{A_1}}  & A_0\ar[d]^{\mathtt{a}_0}  \\  A_{M_1} \ar[r]_{\mathtt{a}} & TM_0, } \] 
where $\mathtt{a}_0$ is the anchor of $A_0$. So we have that $T_x \mathcal{O}'=\mathtt{a}_1(\ker \mathtt{s}_{A_1}|_x)\oplus \mathtt{a}_0(A_0|_x)$ if $x\in \mathcal{O} $.

It is known that $C:=\ker \mathtt{s}_{A_1}|_{M_0}$ is a Lie algebroid over $M_0$ called the {\em core Lie algebroid of $A_1 \rightrightarrows A_0$} \cite{macdou}. Let us recall its structure maps: $\Gamma (C)$ is identified with the space of right invariant sections of $A_1$ as follows: to $X\in \Gamma (C)$ we associate the section $p\mapsto X^r(p)=\mathtt{m}_{A_1}(X(\mathtt{t}(p)),0_p)$ for all $p\in M_1$. By the compatibility of the bracket with the multiplication on $A_1$, the space of right invariant sections is closed under the bracket of $A_1$. Since $\mathtt{a}_1$ preserves the bracket and it takes right invariant sections of $A_1$ to right invariant vector fields on $M_1$, the map 
\[ \phi:=\mathtt{a}_1|:C \rightarrow A_{M_1} \] is a Lie algebroid morphism. As a consequence, $C$ is a Lie algebroid over $M_0$ with anchor $\mathtt{a}\circ \mathtt{a}_1|_C$.

Let $X$ be a section of $C$, define $\rho(X)\in \mathfrak{X}(\mathcal{O})$ as the vector field $(\mathtt{a}\circ \phi)(X)|_{\mathcal{O}  }=(\mathtt{a}_0\circ \mathtt{t}_{A_1})(X)|_{\mathcal{O} }  $. We have that $\rho$ defines an action Lie algebroid $i^*C$ over $\mathcal{O} $, where $i: \mathcal{O} \hookrightarrow M_0$ is the inclusion. In fact, take $X,Y\in \Gamma (C)$, then $\mathtt{t}_{A_1}|_p(X^r|_p)=\mathtt{t}_{A_1}|_{\mathtt{t}(p)}(X)$ for all $p\in G_1$. But $\mathtt{t}_{A_1}$ is a Lie algebroid morphism so we get that 
\[ \mathtt{t}_{A_1}|_p([X^r,Y^r]|_p)=[\mathtt{t}_{A_1}(X^r),\mathtt{t}_{A_1}(Y^r)]|_{\mathtt{t}( p)} \] 
for all $p\in M_1 $. Since $[X,Y]=[X^r,Y^r]|_{M_0}$ by definition, we get that 
\[ \mathtt{t}_{A_1}|_x([X,Y]|_x)=[\mathtt{t}_{A_1}(X),\mathtt{t}_{A_1}(Y)]|_x \] 
for all $x\in M_0$. Therefore, $\rho([X,Y])=\mathtt{a}_0([\mathtt{t}_{A_1}(X),\mathtt{t}_{A_1}(Y)])|_{\mathcal{O}} =[\rho(X ),\rho(Y)]$ since $\mathtt{a}_0$ commutes with the Lie brackets. 

By construction, $\phi$ induces a Lie algebroid morphism $\phi':i^*C \rightarrow A_{M_1}$ over $i$. Indeed, let $\sum_a f^a\otimes X_a$, $\sum_b g^b\otimes Y_b$ be elements of $C^\infty(\mathcal{O})\otimes_{C^\infty(M_0)} \Gamma (C)\cong \Gamma (i^*C)$. Then 
\begin{align*}  \phi'(\sum_a f^a\otimes X_a)=\sum_a f^a \otimes \phi(X_a) , \quad 
\phi'(\sum_b g^a\otimes Y_b)=\sum_b g^b \otimes \phi(Y_b) \end{align*} are elements of $C^\infty (\mathcal{O} )\otimes_{C^\infty(M_0)} \Gamma (A_{M_1}) \cong \Gamma (i^* A_{M_1})$. Hence we have that 
\begin{align*}  &\phi'([\sum_a f^a\otimes X_a,\sum_b g^b\otimes Y_b])=\sum_{a,b}f^ag^b\otimes [\phi(X_a),\phi(Y_b)]+ \\
&+\sum_{a,b}f^a L_{\rho(X_a)} g^b \otimes \phi(Y_b)-\sum_{a,b}g^bL_{\rho(Y_b)}f^a\otimes \phi(X_a); \end{align*}  
and so $\phi'$ is a Lie algebroid morphism, see \cite{moeint} for the definition of Lie algebroid morphism that we used.

The image $\phi'(i^*C)$ is a Lie subalgebroid of $A_{M_1}$ as long as it is a smooth subbundle. This follows from the following dimension counting: 
\[ \dim \mathcal{O}'_x=\rank \mathtt{a}_1|_{A_1|_x}= \rank \mathtt{a}_1|_{C_{x}}+\rank \mathtt{a}_0|_{x} \] 
for all $x\in \mathcal{O} $. The map $\mathtt{a}_1$ has constant rank over $\mathcal{O} $ and the term $\rank \mathtt{a}_0|_{x}$ is constant over $\mathcal{O} $ so $\phi'_x$ is of rank equal to $ \rank \mathtt{a}_1|_{C_{x}}$ which is constant for $x\in\mathcal{O} $. Since the anchor of $i^*C$ restricted to $\ker \phi'$ vanishes, we have that the quotient $B:=i^*C/\ker \phi'$ inherits a Lie algebroid structure over $\mathcal{O} $. Then $\phi'$ induces an injective Lie algebroid morphism $\psi:B \hookrightarrow A_{M_1}$ over $i$. 

{\em Step 2}. Since $\mathcal{O} \hookrightarrow M_0$ is an immersed submanifold and $B \hookrightarrow A_{M_1}$ is a Lie subalgebroid, Lie's second theorem implies that $B$ is integrable and the morphism $\psi: B \hookrightarrow A_{M_1}$ is integrable by a Lie groupoid morphism $\Psi:H \rightarrow M_1 $ which is an immersion, see \cite{moeint}. We claim that $\widehat{\mathcal{O}}=\Psi(H)$. In fact, $T{\mathcal{O}' }$ fits into the fiberwise exact sequence
\[ \xymatrix{0 \ar[r] & T \mathcal{O}'\cap \ker T \mathtt{s} \ar[r] & T \mathcal{O}'\ar[r]^{T \mathtt{s}}&  T \mathcal{O} \ar[r] & 0} \]   
and $T\Psi(H)$ spans both $(T \mathcal{O}'\cap \ker T \mathtt{s})|_{\widehat{\mathcal{O}}}$ and $T \mathcal{O}$ by construction. Since $\mathcal{O}'$ is weakly embedded in $M_1$ \cite{leeman}, $\Psi $ is also smooth with respect to the immersed submanifold structure on $\mathcal{O} $. Then $\widehat{\mathcal{O}}$ is an open subset of $\mathcal{O}$ thanks to the fact that $\Psi$ is an immersion and $\dim H=\dim \mathcal{O} $. Therefore, $\widehat{\mathcal{O}}$ is also an immersed weakly embedded submanifold. Finally, $\Psi$ is a groupoid morphism which is injective on its base so its image is a subgroupoid of $M_1$. Therefore, the result holds.    
\end{proof}
\begin{rema} It is natural to ask whether the groupoid structure that we find on $\widehat{\mathcal{O} }$ extends to the whole $A_1$-orbit $\mathcal{O}'$. We do not know if that is the case. Let us note that, thanks to Proposition \ref{orbdougro}, a counterexample would have to come to from an LA-groupoid which does not admit any source-connected double groupoid integrating it, even in a weak sense as the Weinstein groupoids integrate general Lie algebroids \cite{crarui}. \end{rema} 
\section{Applications}
\subsection{Symplectic leaves of Poisson groupoids}
\begin{defi}[\cite{weicoi}] A {\em Poisson groupoid} is a Lie groupoid $M_1\rightrightarrows M_0$ with a Poisson structure on $M_1$ such that the graph of the multiplication map is a coisotropic submanifold\footnote{Let $M$ be a Poison manifold with Poisson tensor $\pi$. A submanifold $C$ of $M$ is coisotropic if $\pi^\sharp(T^\circ C)\subset TC$, where $T^\circ C$ is the annihilator of $TC$. } of $M_1\times M_1\times \overline{M}_1$, where $\overline{M}_1$ denotes $M_1$ with the opposite Poisson structure. \end{defi}
A Poisson groupoid over a point is a {\em Poisson group} \cite{driham}; a Poisson groupoid whose Poisson bracket is nondegenerate, and hence induced by a symplectic form, is a {\em symplectic groupoid} \cite{weisygr,karpoi}.
The infinitesimal counterpart of a Poisson groupoid is provided by the following objects.
\begin{defi}[\cite{macxu}] Consider a Lie algebroid $A$. If $A^*$ possesses also a Lie algebroid structure such that its differential $d_*$ satisfies
\[ d_*[X,Y]=[d_*X,Y]+[X,d_*Y], \] 
	for all $X,Y\in\Gamma(A)$, then we call $(A,A^*)$ a {\em Lie bialgebroid}. \end{defi}
The cotangent bundle of a Poisson manifold is canonically endowed with a Lie algebroid structure. In the case of a Poisson groupoid $M_1 \rightrightarrows M_0$ such a Lie algebroid structure is compatible with the cotangent groupoid structure thus producing an LA-groupoid. In fact, it follows from equation \eqref{eq:cotgro} that the canonical symplectic form on $T^*M_1$ makes $T^*M_1\rightrightarrows A_{M_1}^*$ into a symplectic groupoid; from this fact it follows that the conormal bundle of the graph $\Gamma$ of $\mathtt{m}$ in $G_1^3$ coincides with the bundle \[E=\{(\xi,\eta,-\hat{\mathtt{m}}(\xi,\eta)): \text{$\xi$, $\eta\in T^*M_1$ composable}\}\subset (T^*M_1)^3 \] 
since $E\subset (T\Gamma)^\circ$ and both vector bundles have the same rank. Using these observations it is not difficult to prove the following.
\begin{thm}[\cite{macxu}]\label{poimap} Let $\Pi$ be a bivector field on a Lie groupoid $M_1 \rightrightarrows M_0$. The map $\Pi^\sharp :T^*{M}_1\rightarrow T{M}_1$ is a groupoid morphism if and only if the graph of $\mathtt{m}$ in $(M_1^3,\Pi\oplus \Pi\ominus \Pi)$ is coisotropic.
\end{thm}
As a consequence of Theorem \ref{poimap}, if $(M_1 , \Pi)$ is a Poisson groupoid, then its cotangent groupoid with the Lie bracket induced by $\Pi$ is an LA-groupoid. 

Let $A$ be the Lie algebroid of a Poisson groupoid $M_1 \rightrightarrows M_0$. We have that $M_0$ inherits a unique Poisson structure such that the target map $\mathtt{t}:M_1 \rightarrow M_0 $ is a Poisson morphism \cite{weicoi}. The symplectic leaves of this Poisson structure are contained in the connected components of the intersection of the $A$ and $A^*$-orbits \cite{macxu}. Let $\mathcal{O}$ be an $A^*$-orbit, we have that ${\mathcal{O}  }$ is a Poisson submanifold of $M_0$, being a union of symplectic leaves. Theorem \ref{orblagr} applied to $T^*M_1 \rightrightarrows A^*$ gives us a symplectic groupoid structure on $\widehat{\mathcal{O}  } \rightrightarrows {\mathcal{O}  }$ which integrates this Poisson structure.
\begin{prop}\label{cotpoigro} Let $M_1\rightrightarrows M_0$ be a Poisson groupoid with tangent Lie bialgebroid $(A,A^*)$ and let $\mathcal{O}\subset M_0 $ be an $A^*$-orbit. Then $\widehat{\mathcal{O}  }$ is an immersed subgroupoid of $M_1$ and it inherits a symplectic groupoid structure which integrates the Poisson structure on $\mathcal{O}$. \end{prop} 
\begin{proof} The symplectic structure on $\widehat{\mathcal{O}  }$ is automatically multiplicative: if we see it as a nondegenerate Poisson bracket, then we have that $\text{graph}(\mathtt{m}_G)\cap(\widehat{\mathcal{O}  } \times \widehat{\mathcal{O}  } \times \overline{\widehat{\mathcal{O}  }})=\text{graph}(\mathtt{m}_{\widehat{\mathcal{O}  }})$ is coisotropic in $\widehat{\mathcal{O}  } \times \widehat{\mathcal{O}  } \times \overline{\widehat{\mathcal{O}  }}$.  

Now we just have to compare the Poisson structure on $\mathcal{O}  $ induced by the target map of $\widehat{\mathcal{O}  }$ with the Poisson structure on $\mathcal{O}  $ induced as a submanifold of $M_0$. We have that $i:\widehat{\mathcal{O}  } \hookrightarrow M_1$ is a Poisson morphism and hence so it is the composition $\mathtt{t}\circ i:\widehat{\mathcal{O}  } \rightarrow M_0$. But $\mathtt{t}\circ i$ is the target map of $\widehat{\mathcal{O}  } \rightrightarrows {\mathcal{O}  }$ so both Poisson structures on ${\mathcal{O}  }$ coincide. \end{proof}\subsection{Orbits of Lie 2-algebra actions on Lie groupoids} The symmetries of Lie groupoids are described by the following objects.
\begin{defi}[\cite{brsp,grp}] A {\em Lie 2-group} is a groupoid object in the category of Lie groups. \end{defi} A groupoid in the category of (Lie) groups is the same thing as a group in the category of (Lie) groupoids. This means that, given a Lie 2-group, $\mathcal{G}:G_1\rightrightarrows G_0$, the multiplication maps $G_i\times G_i\rightarrow G_i$ and the group inversion map $G_i\rightarrow G_i$, $i=0,1$, determine Lie groupoid morphisms, in the first case from $\mathcal{G}\times \mathcal{G}$ to $\mathcal{G} $ and in the second from $\mathcal{G}$ to itself.

\begin{defi}[\cite{bacr}] A groupoid in the category of Lie algebras is called a {\em strict Lie 2-algebra}. An abelian Lie 2-algebra is called a {\em 2-vector space}. \end{defi}

It is straightforward to see that the Lie functor establishes an equivalence of categories between the categories of 1-connected Lie 2-groups and Lie 2-algebras.

Lie 2-algebras are equivalent to crossed modules of Lie algebras.
\begin{defi}[\cite{bacr}] Let $\phi:\mathfrak{h}\rightarrow\mathfrak{g}$ be a Lie algebra morphism and let $\mathfrak{g}$ act on $\mathfrak{h}$ by derivations in such a way that for all $x\in \mathfrak{g}$ and $a,b\in\mathfrak{h}$
\begin{align*} 
\phi(x\cdot a)&=[x,\phi(a)] \\
\phi(a)\cdot b&=[a,b]; 
\end{align*}
then this structure is called a {\em differential crossed module or a crossed module of Lie algebras}. The Lie 2-algebra associated to $\phi$ is the action groupoid $\mathfrak{G} :\mathfrak{h}\rtimes \mathfrak{g} \rightrightarrows \mathfrak{g}$ with action map $\mathtt{t} (a,x)=\phi(a)+x$, for $a\in \mathfrak{h}$, $x\in \mathfrak{g}$. Here $\mathfrak{h}\rtimes \mathfrak{g}$ is the semidirect product Lie algebra given by the $\mathfrak{g}$-action. \end{defi}
Let $M_1 \rightrightarrows M_0$ be a Lie groupoid and let $\mathfrak{G} =\mathfrak{h} \rtimes \mathfrak{g} \rightrightarrows \mathfrak{g} $ be a Lie 2-algebra. Suppose that we have an infinitesimal action $ \mathfrak{G} \rightarrow \mathfrak{X}(M_1)$ such that the induced vector bundle map $\mathtt{a}:\mathfrak{G} \times M_1 \rightarrow TM_1$ is a groupoid morphism with the product groupoid structure on $\mathfrak{G} \times M_1$. Then the action Lie algebroid $\mathfrak{G} \times M_1$ becomes an LA-groupoid over $M_1$. As another corollary of Theorem \ref{orblagr} we get the following result.
\begin{prop}\label{orblie2alg} Let $\mathcal{O}\subset M_0$ be a $\mathfrak{g}$-orbit, then $\widehat{\mathcal{O} }$ is a Lie subgroupoid of $M_1$ over $\mathcal{O} $. \qed \end{prop} 
\subsection{Dressing orbits of Poisson 2-groups} Now se shall see a family of examples that illustrate both Proposition \ref{cotpoigro} and Proposition \ref{orblie2alg}.
\begin{defi} Consider a Lie 2-algebra $\mathfrak{G}:\mathfrak{h}\rtimes \mathfrak{g} \rightrightarrows \mathfrak{g}$ associated to the differential crossed module $\phi:\mathfrak{h}\rightarrow \mathfrak{g}$. The map $-\phi^*:\mathfrak{g}^*\rightarrow \mathfrak{h}^*$ produces an action groupoid (given by the induced action by translations) we shall denote $\mathfrak{G}^*$; the source and target maps of $\mathfrak{G}^*$ have to be reversed in order to have the pairing $\mathfrak{G}^*\times \mathfrak{G}\rightarrow \mathbb{R}$ as a groupoid morphism, that is, for all $(\alpha,\theta)\in(\mathfrak{h}\rtimes \mathfrak{g})^*$:
\begin{align*} &\mathtt{s}(\alpha,\theta)=\alpha-\phi^*(\theta), \quad \mathtt{t} (\alpha,\theta)=\alpha. \end{align*} \end{defi} 
	\begin{defi}\label{coaact} The coadjoint action of a Lie 2-group $\mathcal{G}:H\rtimes G \rightrightarrows G$ on $\mathfrak{G}^*$ is the usual (right) coadjoint action of $H\rtimes G$ on $\left(\mathfrak{h}\rtimes \mathfrak{g}\right)^*$ at the level of arrows and, at the level of objects, it is the dual of the $G$-action on $\mathfrak{h}$. \end{defi} \begin{rema} The coadjoint action above is multiplicative in the sense that the action map $ \mathfrak{G}^* \times\mathcal{G} \rightarrow \mathfrak{G}^*  $ is a groupoid morphism. \end{rema} 

\begin{defi}[\cite{poi2gro}] A Poisson 2-group is a Lie 2-group which is a Poisson groupoid and whose group of arrows is a Poisson group with the same Poisson bivector.
\end{defi}
\begin{exa} The pair groupoid of a Poisson group is a Poisson 2-group. \end{exa}
\begin{exa}\label{coaorb2gro} Let $\mathcal{G} $ be a Lie 2-group over $G$. Then the dual of its Lie 2-algebra $\mathfrak{G}^*$ is a Poisson 2-group with respect to the vector space addition as group structure and with the classical Lie-Poisson bracket as its Poisson structure. This follows from Theorem \ref{poimap} by observing that the coadjoint action $\mathfrak{G}^* \times  \mathcal{G}\rightarrow \mathfrak{G}^*$ of Definition \ref{coaact} is multiplicative and hence the contraction map $\Pi^\sharp:T^* \mathfrak{G}^*\cong \mathfrak{G}^* \times \mathfrak{G}  \rightarrow T \mathfrak{G}$ is a groupoid morphism since it is just the associated infinitesimal action. \end{exa} 
The Lie algebroid structure induced by the Poisson bracket on the cotangent bundle of a Poisson group is isomorphic to an infinitesimal Lie algebra action called {\em dressing action} \cite{semdres}. If $\mathcal{G} $ is a Poisson group, we have that the cotangent groupoid $T^* \mathcal{G} $ is isomorphic to the product groupoid $\mathcal{G} \times \mathfrak{G}^*$ and, as a consequence of the previous observation and Theorem \ref{poimap}, there is a Lie 2-algebra structure on $\mathfrak{G}^*$ in such a way that the dressing action $\Pi^\sharp: T^* \mathcal{G} \cong \mathcal{G} \times \mathfrak{G}^* \rightarrow T \mathcal{G} $ is a groupoid morphism. Proposition \ref{cotpoigro} implies the following: 
\begin{coro}\label{orbpoi2gro} Let $\mathcal{G} $ be a Poisson 2-group over $G$ with tangent Lie bialgebroid $(A,A^*)$ and let $\mathcal{O}$ be an $A^*$-orbit on $G$. Then $\widehat{\mathcal{O} }$ is a symplectic groupoid. \qed\end{coro}

Now we shall see a situation in which the cotangent groupoid of a Poisson 2-group is integrable by a double Lie groupoid. Suppose that a Lie 2-algebra $\mathfrak{G}$ is such that there is also a Lie 2-algebra structure on the 2-vector space $ \mathfrak{G}^*$ with the property that $(\mathfrak{G},\mathfrak{G}^*)$ is a usual Lie bialgebra. Then we say that $(\mathfrak{G},\mathfrak{G}^*)$ is a {\em Lie 2-bialgebra} \cite{poi2gro}. The classical correspondence between Lie bialgebras and 1-connected Poisson groups implies, in particular, that 1-connected Poisson 2-groups are classified by Lie 2-bialgebras \cite{poi2gro}. If $(\mathfrak{G},\mathfrak{G}^*)$ is a Lie 2-bialgebra, it is immediate to check that the double $\mathfrak{G}\oplus \mathfrak{G}^*$ with its classical Lie bracket \cite{driham,luphd} is also a Lie 2-algebra.

Suppose that $\mathcal{G} $ and $\mathcal{G}^*$ are 1-connected Poisson 2-groups which correspond to the Lie 2-bialgebra $(\mathfrak{G},\mathfrak{G}^*)$. Then there are Lie 2-group morphisms $i:\mathcal{G} \rightarrow \mathcal{D} $ and $j:\mathcal{G}^* \rightarrow \mathcal{D}$, where $\mathcal{D}$ is the 1-connected integration of the double $\mathfrak{G}\oplus \mathfrak{G}^*$. Let us recall the integration of the Poisson structures on Poisson groups given in \cite{luwei2}. We take 
		\begin{align*} \mathcal{S} =\{(g,u,v,h)\in \mathcal{G}   \times \mathcal{G}^* \times \mathcal{G}^* \times \mathcal{G} : i(g)j(u)=i(v)j(h)\}. \end{align*} 
The source and target of this groupoid are the projections to $\mathcal{G}  $ and the multiplication is given by
\begin{align*}  \mathtt{m}( (a,u,v,b),(b,u',v',c))=(a,uu',vv',c); \end{align*} 
see \cite{luphd} for a description of the symplectic form on $\mathcal{S} $. Since $i$ and $j$ are Lie 2-group morphisms, we have that $\mathcal{S} $ is a double Lie groupoid with sides $\mathcal{G} $ and $S_0= \{(g,u,v,h)\in \mathcal{G}_0   \times \mathcal{G}^*_0 \times \mathcal{G}^*_0 \times \mathcal{G}_0 : i(g)j(u)=i(v)j(h)\}$ over $\mathcal{G}_0$, where $\mathcal{G}_0$ and $\mathcal{G}_0^*$ denote the groups of units of the corresponding Lie 2-groups. So in this situation we can apply Proposition \ref{orbdougro} instead of Theorem \ref{orblagr} to deduce that the $\mathcal{S}$-orbits that pass through the units in $\mathcal{G} $ are symplectic groupoids, which is a stronger statement than the one provided by Corollary \ref{orbpoi2gro}.  
\begin{rema} The double Lie groupoid $\mathcal{S} $ is actually a {\em double symplectic groupoid} \cite{luwei2}, i.e. its symplectic form is multiplicative with respect to both groupoid structures on $ \mathcal{S} $. This means that also $S_0 \rightrightarrows \mathcal{G}_0$ is a Poisson groupoid. Therefore, Corollary \ref{cotpoigro} gives us another family of symplectic groupoid as subgroupoids of $S_0 \rightrightarrows \mathcal{G}_0$. It would be interesting to know whether Poisson 2-groups are integrable by double symplectic groupoids in general. \end{rema} 
\begin{exa} In the case of Example \ref{coaorb2gro}, Proposition \ref{orbdougro} implies that the whole coadjoint orbit of a unit is a subgroupoid of $\mathfrak{G}^*$ and hence it is a symplectic groupoid. In fact, as we have seen, the infinitesimal coadjoint action on $\mathfrak{G}^*$ is integrable by the double Lie groupoid 
\[  \xymatrix{ \mathfrak{G}^* \times \mathcal{G}\ar@<-.5ex>[r] \ar@<.5ex>[r]\ar@<-.5ex>[d] \ar@<.5ex>[d]& \mathfrak{G}^*\ar@<-.5ex>[d] \ar@<.5ex>[d] \\ \mathfrak{h}^* \times G \ar@<-.5ex>[r] \ar@<.5ex>[r] & \mathfrak{h}^*  }\]
given by the coadjoint action. The Poisson structures that these symplectic groupoids integrate are the $G$-orbits in $\mathfrak{h}^*$ and these can be quite varied, for instance: 
\begin{itemize} \item if $\mathfrak{G}^*$ is isomorphic to the pair groupoid on the dual of a Lie algebra $\mathfrak{G}^*\cong (\mathfrak{g}^* \times \mathfrak{g}^* \rightrightarrows \mathfrak{g}^*)$, then the symplectic groupoids that we get in this manner are the pair groupoids over the coadjoint orbits in $\mathfrak{g}^*$ endowed with their canonical symplectic structures, \item if $\mathfrak{G}^*$ is given by a differential crossed module defined by the zero map $0:\mathfrak{h} \rightarrow \mathfrak{g} $, then the coadjoint orbits of the units in $\mathfrak{G}^*$ integrate the null Poisson structures on the $G$-orbits in $\mathfrak{h}^*$. \item Take $V$ a symplectic vector space and let $\phi:\mathfrak{h} \rightarrow \mathfrak{g} $ be the differential crossed module associated to the Heisenberg Lie algebra $\mathfrak{h}=V \times \mathbb{R}$, where $\mathfrak{g}$ is the Lie algebra of derivations of $\mathfrak{h} $ and $\phi$ is defined by $u\mapsto \ad_u$. In this case we can integrate $\mathfrak{h} $ with the Heisenberg group $H$ and we can integrate $\mathfrak{g} $ with $G=\text{Aut}(\mathfrak{h} )$. So we can see that there are three kinds of $G$-orbits in $\mathfrak{h}^*$: two open orbits $\{(\alpha ,\mu)\in \mathfrak{h}^* |\mu\neq 0\}$ isomorphic to $V \times (0,\infty)$ with its canonical Lie-Poisson structure and the orbits  of the form $ \mathcal{O} \times \{0\}\subset \mathfrak{h}^*$, where $\mathcal{O}$ is an orbit of the linear symplectic group of $V$, which are endowed with the zero Poisson bracket. \end{itemize}   \end{exa}  
\subsection*{Acknowledgements} The author thanks CNPq for the financial support and D. Mart\'inez Torres for fruitful conversations. This work is part of the author's PhD thesis written under the helpful guidance of H. Bursztyn. The author also thanks the referee whose comments and suggestions significantly improved this work.

\printbibliography
\end{document}